\documentclass[12pt,reno]{amsart}  
\usepackage{amssymb}
\usepackage{amsmath}
\usepackage{enumitem}
\usepackage{mathrsfs}
\usepackage[english]{babel}
\usepackage[all]{xy}
\usepackage{hyperref}
\usepackage{marginnote}

\usepackage{stmaryrd}

\usepackage{fullpage}

\RequirePackage{mathrsfs} %\let\mathcal\mathcal

\newtheorem{theorem}{Theorem}[section]
\newtheorem{lemma}[theorem]{Lemma}

\theoremstyle{definition}
\newtheorem*{ack}{Acknowledgements}
\newtheorem*{con}{Conventions}
\newtheorem{remark}[theorem]{Remark}
\newtheorem{example}[theorem]{Example}
\newtheorem{definition}[theorem]{Definition}

\numberwithin{equation}{section} \numberwithin{figure}{section}

\DeclareMathOperator{\Spec}{Spec}

\DeclareMathOperator{\im}{Im}

\newcommand{\set}[2]{\left\{ #1 \; \middle| \; #2 \right\}}

\newcommand\ZZ{\mathbb{Z}}

\newcommand\QQ{\mathbb{Q}}

\newcommand\CC{\mathbb{C}}

\newcommand{\et}{\textrm{\'{e}t}}

\title[The Hilbert property for arithmetic schemes]{The Hilbert property for arithmetic schemes}

\author{Cedric Luger}
\address{Cedric Luger\\
Institut f\"{u}r Mathematik\\
Johannes Gutenberg-Universit\"{a}t Mainz\\
Staudingerweg 9, 55099 Mainz\\
Germany.}
\email{celuger@uni-mainz.de}

\subjclass[2020]
{14G99 %Arithmetic problems ("None of the above, but in this section")
(14G05,  %Rational points
14G40)}		%Arithmetic varieties and schemes; Arakelov theory; heights [See also 11G50, 37P30]


\begin{document}

\begin{abstract}  
We extend the usual Hilbert property for varieties over fields to arithmetic schemes over integral domains by demanding 
the set of near-integral points (as defined by Vojta) to be non-thin. We then generalize results of Bary-Soroker--Fehm--Petersen and Corvaja--Zannier by proving  several structure results related to products and finite \'{e}tale covers of arithmetic schemes with the Hilbert property.
\end{abstract}

\maketitle

\thispagestyle{empty}

  \section{Introduction}
  Recall that, for an integral variety $X$ over a field $K$, a subset $S\subseteq~X(K)$ is called \emph{thin} (in $X$ or in $X(K)$) if there exist an integer $n$, integral $K$-varieties $Y_1,\ldots,Y_n$ and dominant generically finite separable morphisms $\pi_i\colon Y_i\to X$ of degree at least two such that $S\setminus \bigcup_{i=1}^n \pi_i(Y_i(K))$ is not dense in $X$. One says that $X$ satisfies the \textit{Hilbert property over $K$} if $X(K)$ is not thin.
  The Hilbert property is closely related to the inverse Galois problem for $\QQ$ (see \cite[\S4]{SerreTopicsGalois}) and studied in
\cite{BSFP22,
Coccia,
CTS,
CDJLZ,
CZHP,
Demeio1, Demeio2,
FriedJarden,
GvirtzChenHuang, GvirtzChenMezzedimi,
JavanpeykarProductOfCurves, JavNef,
Lang60, LangDiophantine2,
LoughranSalgado,
NakaharaStreeter,
Streeter,
ZannierDuke}.

 In this paper we study an extension of the Hilbert property to schemes of finite type over integral domains (such as the ring of integers of a number field).
 In order to do so, it is natural to look at integral points (instead of rational points). In fact, as noticed by Vojta \cite[\S 4]{VojtaLangExc}, when studying properties of rational points on varieties over finitely generated fields of characteristic zero, it is more natural to look at ``near-integral points'' (see also \cite[\S 3]{JBook}).
 
 \begin{definition}\label{DefinitionNearIntegral}
 	Let $R$ be an integral domain with fraction field $K$ and let $X$ be a scheme over $R$. A rational point $x\in X(K)$ is called \emph{near-$R$-integral} if there exists
 	a dense open $U\subseteq \Spec R$ with complement of codimension at least two such that $x$ is contained in the image of $X(U)\to X(K)$. The set of near-$R$-integral points of $X$ is denoted $X(R)^{(1)}$.
 \end{definition}
 
 Our extension of the Hilbert property now reads as follows.
  
  \begin{definition}\label{DefinitionHPoverR}
  	Let $R$ be an integral domain with fraction field $K$. An integral scheme $X$ that is dominant, separated and of finite type over $R$ satisfies the \textit{Hilbert property over $R$} if $X(R)^{(1)}$ is not thin in $X(K)=X_K(K)$.
  \end{definition}
  
  We refer to an integral finite type separated dominant scheme over an integral domain $R$ as an \textit{arithmetic scheme} over $R$.
  Thus, the above notion extends the usual Hilbert property for varieties over fields to arithmetic schemes over integral domains.
  In this paper, our focus will be on arithmetic base rings such as the ring of $S$-integers in a number field $K$, but also certain finitely generated extensions of such rings.

\begin{example}
	Let $R$ be a $\mathbb{Z}$-finitely generated integral domain with fraction field $K$ of characteristic zero, and let $n\geq 1$ be a natural number.
	For $f\in K(t_1,\ldots,t_n)[X]$ irreducible, set
	\begin{align*}
		U_{f,K}
		& = \left\{
		    (t'_1,\ldots,t'_n) \in \mathbb{A}^n_K(K) \;\middle|\; f(t'_1,\ldots,t'_n,X) \text{ is defined and irreducible in } K[X]
		    \right\} \\
		& \subseteq \mathbb{A}_K^n(K)
	.\end{align*}
	A \textit{Hilbert subset} of $\mathbb{A}^n_K$, as defined in \cite{LangDiophantine2} and \cite{FriedJarden}, is the intersection of finitely many dense open subsets of $\mathbb{A}^n_K$ and finitely many sets $U_{f,K}$.
	In particular, the complement of a Hilbert subset in $\mathbb{A}^n_K(K)$ is a thin subset.
	It is shown in \cite[Ch. 9, Theorem~4.2]{LangDiophantine2} and \cite[Proposition 13.4.1]{FriedJarden} that every Hilbert subset of $\mathbb{A}^n_K$ contains an $R$-integral point,
	i.e., the image of $\mathbb{A}^n_R(R) \to \mathbb{A}^n_K(K)$ is not thin.
	Therefore, the arithmetic scheme $\mathbb{A}^n_R$ has the Hilbert property over $R$. (This also follows for all $n\geq2$ from the Hilbert property of $\mathbb{A}^1_R$ over $R$ and Theorem \ref{TheoremProductHP}.) Of course then $\mathbb{P}_R^n$ also has the Hilbert property over $R$, as it contains $\mathbb{A}^n_R$ as a dense open subscheme.
	
	However, if $p$ is a prime number and we let $D$ denote the image of an element $x \in \mathbb{A}^1(\mathbb{Z})$, then $\mathbb{A}^1_{\mathbb{Z}[1/p]}\setminus D$ does not have the Hilbert property over $\mathbb{Z}[1/p]$.
	Indeed, $\mathbb{A}^1_{\mathbb{Z}[1/p]} \setminus \{0\} \cong \mathbb{A}^1_{\mathbb{Z}[1/p]} \setminus D$ has non-trivial finite \'{e}tale covers, so that the claim follows from Theorem~\ref{TheoremEtaleGroupTrivial} below.
\end{example}

\begin{remark}
Let $X$ be a smooth projective variety over a number field $K$ with only finitely many $L$-points for each number field $L/K$. Lang asked whether this finiteness \emph{persists} over all finitely generated fields of characteristic zero (see \cite[p.~202]{Lang}). Lang's philosophy is that the right setting to consider finiteness questions for rational points is that of varieties over finitely generated fields of characteristic zero; see the ``Persistence Conjecture'' for a precise conjectural statement (e.g., \cite[Conjecture~1.5]{JAut}).  Similarly, the right setting to consider density questions for rational points (such as the Hilbert property) is that of varieties over finitely generated fields of characteristic zero, as we pursue here.
\end{remark}

Recall that Bary-Soroker--Fehm--Petersen \cite{BarySoroker} proved that the product of two varieties with the Hilbert property also satisfies the Hilbert property.
Their result gives an affirmative answer to an old question of Serre (asked in \cite[\S3.1]{SerreTopicsGalois}).
  Note that Bary-Soroker--Fehm--Petersen's theorem can also be deduced from \cite[Lemma~8.12]{HarpazWittenberg}, which builds on \cite[Lemma~3.12]{Wittenberg}.

 Our first main result is the following extension of Bary-Soroker--Fehm--Petersen's product theorem.
 
 \begin{theorem}\label{TheoremProductHP}
 		Let $R$ be an integral domain and let $X$ and $Y$ be arithmetic schemes over $R$. If $X$ and $Y$ satisfy the Hilbert property over $R$, then so does $X\times_R Y$.
 \end{theorem}

  To prove Theorem \ref{TheoremProductHP} we follow closely Bary-Soroker--Fehm--Petersen's \cite{BarySoroker} proof
  for varieties.
  However, we will have to prove a slight improvement of their fibration theorem (see Theorem~\ref{TheoremFibrationIntegral} below) to deal with the more general situation of products of arithmetic schemes.
Bary-Soroker--Fehm--Petersen's product theorem for the Hilbert property over $K$ is a direct consequence of their following more general fibration theorem.

\begin{theorem}[{\cite[Theorem 1.1]{BarySoroker}}]\label{TheoremBarySorokerProductHP}
	Let $X \to S$ be a dominant morphism of integral varieties over a field $K$.
	Assume that the set of $s \in S(K)$ for which the fibre $X_s$ is integral and satisfies the Hilbert property over $K$
is not thin. Then $X$ satisfies the Hilbert property over $K$.
\end{theorem}

Our proof of Theorem \ref{TheoremProductHP} differs only slightly from the one given in \cite{BarySoroker}, providing a minor improvement of Theorem \ref{TheoremBarySorokerProductHP}.
In fact, Theorem \ref{TheoremBarySorokerProductHP} suffices to conclude that, for $K$-varieties $X$ and~$Y$ satisfying the Hilbert property over $K$, the product $X\times_K Y$ also has the Hilbert property over $K$, but not for the analogue over an integral domain $R$ with fraction field $K$.
The problem in the situation of Theorem \ref{TheoremProductHP} is that the Hilbert property of a fibre
$X_{K,s}$ over $s\in S(R)^{(1)}$ does not give us near-$R$-integral points of $X$, and asking $X_s$ to satisfy the Hilbert property over $R$ is too
strong, because instead of points in $X_s(R)^{(1)}$, we are interested in points in $X(R)^{(1)}$ contained in $X_{K,s}$. This leads to the following natural generalization
of Theorem \ref{TheoremBarySorokerProductHP}.

\begin{theorem}
\label{TheoremFibrationIntegral}
	Let $X \to S$ be a dominant morphism of integral $K$-varieties and
$\Gamma \subseteq X(K)$ a subset. Assume that there exists a set $\Sigma \subseteq S(K)$, not thin in $S$, such that,
for each $s \in \Sigma$, the fibre $X_s$ is integral and $\Gamma \cap X_s$ is not a thin subset of $X_s$. Then $\Gamma$
is not thin in $X$.
\end{theorem}

As Corvaja--Zannier show in \cite{CZHP}, the Hilbert property for varieties over number fields forces to\-po\-lo\-gi\-cal constraints on the variety.
Namely, if $X$ is a normal projective variety with the Hilbert property over a number field $K$, then each finite \'{e}tale cover of $X_{\overline{K}}$
is trivial, i.e., the \'etale fundamental group $\pi_1^{\et}(X_{\overline{K}})$ is trivial.
We follow very closely Corvaja--Zannier's line of reasoning and prove an analogue of their theorem in the setting of arithmetic schemes over regular finitely generated integral domains of characteristic zero.

\begin{theorem}\label{TheoremEtaleGroupTrivial}
	Let $R$ be a regular $\ZZ$-finitely generated integral domain with fraction field $K$ of characteristic zero, and let $X$ be a normal arithmetic scheme over $R$.
If $X$ has the Hilbert property over $R$, then $X$ does not have any non-trivial finite \'{e}tale covers over $R$, i.e.,
any  finite \'etale morphism  $Y \to X$  with $Y$ an arithmetic scheme over $R$ whose generic fibre $Y_K$ is geometrically connected is an isomorphism.
\end{theorem}

As an application of this result,
we extend Corvaja--Zannier's result on the (geometric) algebraic simple connectedness of normal proper varieties with the Hilbert property over number fields to finitely generated fields of characteristic zero.
Here the interpretation of $K$-points on the generic fibre as near-$R$-integral points on a suitable model is crucial (as they do not give integral $R$-points necessarily).

In light of the triviality of the fundamental group of varieties with the Hilbert property, the more refined ``weak-Hilbert property'' was introduced in \cite{CZHP}, and it was shown in \cite{CDJLZ} that abelian varieties with a dense set of $K$-points satisfy the weak-Hilbert property.  Conjecturally, a variety with a dense set of $K$-rational points should have the weak-Hilbert property and even be special in the sense of Campana. In light of Campana's Abelianity Conjecture \cite[Conjecture~7.1]{CampanaOrbifolds}, we expect the topological fundamental group of a variety with a dense set of rational points to be virtually abelian (i.e., to contain a finite index abelian subgroup).
In particular, as  smooth projective special varieties over $\mathbb{C}$ are expected to have a virtually abelian topological fundamental group, we also expect the topological fundamental group $\pi_1(X_{\CC})$ of a smooth projective variety $X$ over a number field $K$ with a dense set of $K$-rational points to be virtually abelian.
Although this is far from known, there has been some progress on its function field analogue in \cite[Theorem~1.15]{JavAlbanese} building on recent progress on Lang's conjecture \cite{JBook, Lang}.

Our last result is the persistence of the Hilbert property of arithmetic schemes under finite flat extensions of finitely generated integral domains of characteristic zero.
\begin{theorem}\label{TheoremFiniteFlatBaseChange}
	Let $R\subseteq S$ be a finite flat extension of $\ZZ$-finitely generated integral domains of characteristic zero,
	and let $X$ be an arithmetic scheme over $R$. If $X$ has the Hilbert property over $R$, then $X_S$ has the Hilbert property over $S$.
\end{theorem}

 \begin{ack}
   We gratefully acknowledge Ariyan Javanpeykar for introducing us to this topic and many helpful discussions.
   We thank the referee for many helpful comments and suggestions.
   This research was funded by the Deutsche Forschungsgemeinschaft (DFG, German Research Foundation) – Project-ID 444845124 – TRR 326.
 \end{ack}
 
 \begin{con}
 If $K$ is a field, then a variety over $K$ is a finite type separated scheme over $K$.
 A field $K$ is said to be \emph{finitely generated} if it is finitely generated over its prime field. 
  If $X\to S$ is a morphism of schemes and $s\in S$, then $X_s$ denotes the scheme-theoretic fibre of $X$ over $s$.
  Throughout this paper, we assume $k$ to be a finitely generated field of characteristic zero.
  If $S$ is an integral scheme, we let $k(S)$ denote its function field.
  We say that a dominant morphism $f \colon X \to S$ of integral varieties is \textit{generically finite} if the induced extension $k(S)\to k(X)$ of function fields is finite,
  and that $f$ is \textit{separable} if $k(S)\to k(X)$ is separable.
  If $f$ is generically finite, then by the \textit{degree} of $f$ we mean the degree of its extension of function fields.
 \end{con}

\section{The proof of Theorems \ref{TheoremProductHP} and \ref{TheoremFibrationIntegral}}

We start with the following generalization of \cite[Lemma 3.2]{BarySoroker}.

\begin{lemma}\label{LemmaNotSubset}
	Let $f \colon X \to S$ be a dominant morphism of normal integral varieties over $K$ and let $\Gamma \subseteq X(K)$ be a subset. Assume that there exists a non-thin set $\Sigma\subseteq S(K)$ such that, for each $s\in \Sigma$, the fibre $X_s$ is integral and $\Gamma \cap X_s$ is not a thin subset of $X_s$. Let $Y_1,\ldots,Y_n$ be integral $K$-varieties and let  $\pi_i \colon Y_i \to X$ be finite \'{e}tale morphisms of degree at least two. Then
	\[
		\Gamma \not \subseteq \bigcup_{i=1}^n \pi_i(Y_i(K))
	.\]
\end{lemma}
\begin{proof}	
	For $i=1,\ldots,n$, consider the composite morphism  $\varphi_i := f \circ \pi_i \colon Y_i \to S$. By
	\cite[Lemma~9]{KollarRat}, there is a dense open subscheme $U_i$ of $S$ and a factorization
	\[
		\varphi_i^{-1}(U_i)
		\overset{g_i}{\longrightarrow}
		T_i
		\overset{r_i}{\longrightarrow}
		U_i
	\]
of $\varphi_i$ such that the morphism $g_i$ has geometrically irreducible fibres, $r_i$ is finite and
\'{e}tale, and $T_i$ is an integral $K$-variety.

Replacing $S$ by $\bigcap_{i=1}^n U_i$ if necessary (and pulling back $X\to S$, $Y_i\to S$ and $\Gamma$ accordingly),
we may assume that $r_i \colon T_i \to S$
is finite \'{e}tale for every $i$.

Set $J = \set{i}{ \deg r_i \geq 2 } \subseteq \{1,\ldots,n\}$ and let $I$ be its complement.

For $s\in \Sigma$ and $i\in \{1,\ldots,n\}$, the morphism $\pi_{i,s} \colon Y_{i,s} \to X_s$ induced by $\pi_i$ is finite étale of the same degree as $\pi_i$ and $Y_{i,s}$ is a reduced scheme over $K$. We have the following commutative diagram.
\[
	\xymatrix{
	Y_{i,s} \ar[r] \ar[d]_{\pi_{i,s}} & Y_i \ar[r]^{g_i} \ar[d]_{\pi_i} \ar[dr]^{\varphi_i} & T_i \ar[d]^{r_i} \\
	X_s \ar[r] & X \ar[r]^f & S
	}
\]
Since $\Sigma$ is not thin, there exists a point
\[
	s \in \Sigma \setminus \bigcup_{i\in J} r_i(T_i(K))
.\]
Therefore, for any $i\in J$, we have $T_{i,s}(K)=\emptyset$ and thus $\bigcup_{i\in J}Y_{i,s}(K)=\emptyset$.
Furthermore, for $i \in I$, the morphism $r_i$ is finite étale of degree one, i.e., an isomorphism. Therefore, the reduced scheme $Y_{i,s}$ is also (geometrically) irreducible, i.e., an integral $K$-variety.
Since $\Gamma\cap X_s$ not thin in $X_s$, there exists an element $x \in \Gamma \cap X_s$ such that
\[
	x \notin \bigcup_{i \in I} \pi_{i,s}(Y_{i,s}(K))
.\]
Since $\bigcup_{i\in J}Y_{i,s}(K)=\emptyset$, we get that
\[
	x \notin \bigcup_{i \in I} \pi_{i,s}(Y_{i,s}(K)) = \bigcup_{i=1}^n \pi_{i,s}(Y_{i,s}(K))
.\]
Thus, $x\notin \bigcup_{i=1}^n \pi_{i}(Y_{i}(K))$, as required.

\end{proof}

Before proving our fibration theorem, we show that it suffices to consider finite surjective coverings (by normal geometrically irreducible varieties) when testing for thinness of a set.

	\begin{lemma}\label{LemmaThinSetWeaker}
	Let $X$ be an integral $K$-variety and $S\subseteq X(K)$. Then $S$ is thin if and only if there exists an integer $n \in \mathbb{N}$ and, for each $i = 1,\ldots,n$, a normal geometrically irreducible $K$-variety $Y'_i$ endowed with a finite surjective separable morphism $\pi'_i \colon Y'_i \to X$ of degree at least two such that the set
	\[
		S \setminus \bigcup_{i=1}^n \pi'_i (Y'_i(K)) \subseteq X
	\]
	is not dense in $X$.
\end{lemma}
\begin{proof}
	We only need to prove the ``only if'' part of the statement.
	Assume that $S \setminus \bigcup_{i=1}^n \pi'_i (Y'_i(K))$ is dense for all finite surjective separable $\pi'_i \colon Y'_i \to X$ as in the statement of the lemma. We then have to show that $S$ is not thin, i.e., given integral $K$-varieties $Y_1,\ldots,Y_n$ and $\pi_i \colon Y_i \to X$ dominant generically finite separable morphisms of degree $\deg \pi_i \geq 2$, we have to show that there exists a dense set of points $x \in S$ of $X$ such that $Y_{i,x}(K) = \emptyset$ for all~$i$.
	
	If any of the $Y_i$ is not geometrically irreducible, then $Y_i(K)$ is not dense in $Y_i$, and therefore $\pi_i(Y_i(K))$ is not dense in $X$, so we may and do assume all $Y_i$ to be geometrically irreducible.
	Since the morphisms $\pi_i$ are separated and of finite type, we can use Nagata's compactification theorem (see \cite{Nagata1, Nagata2}, or \cite{ConradNagataCompactification} for a more modern proof) to obtain an open immersion $Y_i \hookrightarrow Y_i'$ over $X$, where the morphism $\pi_i' \colon Y_i' \to X$ is proper.
	If $Y'_{i,x}(K) = \emptyset$ for some $x\in S$, then this implies $Y_{i,x}(K)	= \emptyset$, so it suffices to show that $S \setminus \bigcup_{i=1}^n \pi'_i (Y'_i(K))$ is dense in $X$.
	Let
	\[
		Y'_i
		\overset{f_i}{\longrightarrow} X_i
		\overset{g_i}{\longrightarrow} X
	\]
	be the Stein factorization of $\pi'_i$, where $f_i$ is proper and has geometrically connected fibers, and $g_i$ is finite. Since $\pi'_i$ is dominant and proper (which implies closed), it is also surjective, so $g_i$ must be surjective too.
	
	If $\eta_i$ denotes the generic point of $X_i$, then $Y'_{i,\eta_i}$ is finite and connected, i.e. an isolated point. Since $X_i$ is normal, $f_i$ is birational by Zariski's Main Theorem.
	Thus, we get
	\[
		\deg(X_i \to X) = \deg(Y'_i \to X) = \deg(Y_i \to X) \geq 2
	.\]

	The $X_i \to X$ are finite surjective separable of degree at least two, so that by our assumption there is a dense set of $K$-points $x\in S$ of $X$ such that none of the $X_{i,x}$ have a $K$-point, hence neither do the $Y'_{i,x}$, which is what we had to show.
\end{proof}

We now prove our fibration theorem.

\begin{proof}[Proof of Theorem \ref{TheoremFibrationIntegral}]
	Let $C\subsetneq X$ be a proper closed subset and let $Y_i$ be integral $K$-varieties, endowed with finite surjective separable morphisms $\pi_i \colon Y_i \to X$ of degree at least two, $i= 1,\ldots,n$ (cf. Lemma \ref{LemmaThinSetWeaker}). We have to prove that $\Gamma \not \subseteq C \cup \bigcup_{i=1}^n\pi_i(Y_i(K))$.

	Since the extensions of function fields induced by $\pi_i$ are separable, there exists a dense open subscheme $W\subseteq X$ such that, for every $i=1,\ldots,n$,
	the restriction $\pi_i^{-1}(W) \to W$ is finite \'{e}tale.
	By this and the openness of the normal locus \cite[6.12.6 and 6.13.5]{EGAIV2}, we may choose a normal dense open subscheme $S'\subseteq S$ and a normal dense open subscheme
	\[
		X' \subseteq f^{-1}(S') \setminus C \subseteq X
	\]
	such that $f$ restricts to a finite \'{e}tale morphism of normal $K$-varieties $f' \colon X'\to S'$.
	
	Set $Y_i' = \pi_i^{-1}(X')\subseteq~Y_i$ and denote the restriction of $\pi_i$ by $\pi_i' \colon Y_i' \to X'$. Then $\pi_i'$ is dominant and generically finite of the same degree as $\pi_i$.
	Furthermore, $\Sigma \cap S'$ is not thin in $S'$ and, for $s\in \Sigma\cap S'$, the fibre $X'_s$ is an open subset of $X_s$, so that it is integral (by our assumption that $X_s$ is integral) and $X'_s \cap \Gamma$ is not thin in $X_s$. By Lemma \ref{LemmaNotSubset}, this gives us
	\[
		\Gamma \cap X' \not\subseteq \bigcup_{i=1}^n \pi_i'(Y_i'(K))
	.\]
	Therefore,
	$
		(\Gamma \cap X') \setminus \bigcup_{i=1}^n \pi_i'(Y_i'(K)) \neq \emptyset
	$.	
	Obviously,
	\[
		(\Gamma \cap X') \setminus \bigcup_{i=1}^n \pi_i'(Y_i'(K)) = (\Gamma\cap X') \setminus \bigcup_{i=1}^n \pi_i(Y_i(K))
	.\]
	Finally, by our definition of $X'$,
	\[
		(\Gamma\cap X') \setminus \bigcup_{i=1}^n \pi_i(Y_i(K)) \subseteq \Gamma \setminus \left( C \cup  \bigcup_{i=1}^n \pi_i'(Y_i'(K)) \right)
	.\]
	We conclude that $\Gamma \not \subseteq C \cup \bigcup_{i=1}^n\pi_i(Y_i(k))$, as required.
\end{proof}

From this fibration theorem we can now deduce the first main theorem.

\begin{proof}[Proof of Theorem \ref{TheoremProductHP}]
	Let $K$ denote the fraction field of $R$ and note that $X_K$ and $Y_K$ are geometrically integral over $K$ (as they are integral and their $K$-rational points are dense), so that $X_K \times_K Y_K$ is geometrically integral over $K$. In particular, $X \times_R Y$ is integral.	
	We apply Theorem~\ref{TheoremFibrationIntegral} to $\Gamma :=(X \times_R Y)(R)^{(1)}$, $\Sigma := Y(R)^{(1)}$ and the projection $X\times_R Y\to Y$.
	Since $Y$ has the Hilbert property over $R$, the set $\Sigma$ is not thin in $Y_K$. Let $s\in \Sigma$ and consider the projection morphism $p\colon X_K \times_K Y_K \to Y_K$. Its scheme-theoretic fibre $p^{-1}(\{s\})$ is isomorphic to $X_K$.
	Since $X(R)^{(1)}$ is not thin in $X_K$, it suffices to show that
	$
		X(R)^{(1)} \subseteq \Gamma \cap p^{-1}(\{s\})
	$
	as subsets of $p^{-1}(\{s\}) \cong X_K$, because then $\Gamma \cap p^{-1}(\{s\})$ is not thin in $p^{-1}(\{s\})$.
	
	Let $x\in X(R)^{(1)}$. By the definition of near-$R$-integral points, there exist dense open subschemes $U\subseteq \Spec R$ and $U'\subseteq \Spec R$ with complements of codimension at least two such that $x \in \im(X(U)\to X(K))$ and $s\in \im(Y(U')\to Y(K))$, respectively.
	Note that $U\cap U'$ is also a dense open of $\Spec R$ with complement of codimension at least two. Clearly, this gives us
	\[
		x \in \im(X(U\cap U')\to X(K)) \text{ and } s\in \im(Y(U\cap U')\to Y(K))
	,\]
	so that there exists a unique dotted arrow making the following diagram of $R$-schemes commute.
	\[
		\xymatrix{
			X & X \times_{R} Y \ar[l] \ar[r] & Y\\
			\Spec K \ar[u]^x \ar[r] & U \cap U' \ar[ul] \ar[ur] \ar@{-->}[u] & \Spec K \ar[l] \ar[u]_s
		}
	\]
	This defines a near-integral point of $X \times_R Y$ whose image in $X_K \times_K Y_K$ lies in the fibre $p^{-1}(\{s\})$, and we are done by Theorem \ref{TheoremFibrationIntegral}.
\end{proof}

\section{The proof of Theorem \ref{TheoremEtaleGroupTrivial}}

The following lemma was stated by Corvaja--Zannier \cite{CZHP} for number fields, but their proof works for finitely generated fields of characteristic zero.

\begin{lemma}[{\cite[Proposition 1.5]{CZHP}}]\label{LemmaCZ}
	Let $X$ and $Y$ be quasi-projective integral varieties over $k$ and let $\pi\colon Y \to X$ be a finite surjective morphism. Let $k'/k$
be a finite field extension and let $T \subseteq Y(k')$ be a set such that $\pi(T) \subseteq X(k)$.
	If $\deg \pi > 1$, then there exist
	quasi-projective integral varieties $Y_1,\ldots,Y_n$ over $k$ and
	finite surjective morphisms $(\pi_i \colon Y_i \to X)_{i=1}^n$
	with $\deg\pi_i > 1$ such that $\pi(T ) \subseteq \bigcup_i \pi_i(Y_i(k))$.
\end{lemma}

We now prove Theorem \ref{TheoremEtaleGroupTrivial} using Lemma \ref{LemmaCZ} and a Chevalley--Weil type lifting argument.

\begin{proof}[Proof of Theorem \ref{TheoremEtaleGroupTrivial}]
	Let $K$ denote the fraction field of $R$ and let $\pi \colon Y\to X$ be a finite \'{e}tale (surjective) morphism of schemes over $R$ with $Y$ integral. Then the base change $\pi_K \colon Y_K \to X_K$ of $\pi$ along $\Spec K \to \Spec R$ is also finite \'{e}tale surjective. By assumption, the set $X(R)^{(1)}$ is not thin in $X_K$.
	
	We follow the proof of the Chevalley--Weil theorem given in \cite[Theorem 7.9]{JBook} (see also \cite[Theorem 3.8]{CDJLZ}) to show that there exists a finite field extension $K'/K$ such that $X(R)^{(1)} \subseteq \pi(Y(K'))$.
	
	For each $x\in X(R)^{(1)}$, choose a dense open $U_x\subseteq \Spec R$ with complement of codimension at least two such that $x$ defines a morphism $x \colon U_x \to X$. Set
	$V_x := U_x \times_X Y$, so that $V_x \to U_x$ is finite \'{e}tale. Since $R$ is regular, by Zariski--Nagata purity of the branch locus \cite[Th\'{e}orème X.3.1]{SGA1}, this morphism extends to a finite \'{e}tale morphism
	\[
		\Spec S_x \to \Spec R
	\]
	of degree $\deg \pi$.
	(Here $S_x$ is the normalization of $R$ in $V_x$.)
	By Hermite's finiteness theorem \cite{smallness}, there exist only finitely many finite \'{e}tale coverings of $\Spec R$ of given degree.
	Therefore, there exists a finite field extension $K'$ of $K$ such that $X(R)^{(1)} \subseteq \pi(Y(K'))$.
	
	Let $X' \subseteq X_K$ be a dense open affine. By the finiteness of $\pi_K$, the preimage $Y' = \pi_K^{-1}(X')$ is also affine.
	In particular, the restriction $Y' \to X'$ of $\pi_K$ is a finite surjective morphism of quasi-projective integral varieties over $K$. Set $T := \pi_K^{-1}(X(R)^{(1)}) \cap Y'(K') \subseteq Y'(K')$, and note that
	\[
		\pi_K(T) = X(R)^{(1)} \cap X' \subseteq X'(K)
	.\]
	Assume for a contradiction that $\deg\pi > 1$. Then, by Lemma \ref{LemmaCZ}, we have
	finitely many quasi-projective integral varieties $Y_i$ over $K$
	and dominant rational maps $\pi_i \colon Y_i \to X'$, each of degree $>1$, such that
	$\pi_K(T ) \subseteq \bigcup_i \pi_i(Y_i(K))$.
	From this it follows that
	$X(R)^{(1)} \setminus \bigcup_i \pi_i(Y_i(K)) \subseteq X_K\setminus X'$
	is not dense in $X_K$. This contradicts the fact that $X$ has the Hilbert property over $R$.
	Thus, $\deg \pi =1$, as required.
\end{proof}

As an application of Theorem \ref{TheoremEtaleGroupTrivial}, we show that it implies that
Corvaja--Zannier's result on the geometric algebraic simple connectedness of normal proper varieties with the Hilbert property over number fields holds over finitely generated fields of characteristic zero. The proof of this extension is similar to Corvaja--Zannier's proof, except that we used a stronger version of the Chevalley--Weil theorem \cite[Theorem 3.8]{CDJLZ} to prove Theorem \ref{TheoremEtaleGroupTrivial}. We will include the proof for the sake of completeness, and for illustrating the relevance of near-integral points.

\begin{theorem}[Corvaja--Zannier]
	If $X$ is a normal proper variety over a finitely generated field $k$ of characteristic zero with the Hilbert property over $k$, then each finite \'{e}tale cover of $X_{\overline{k}}$ is trivial.
\end{theorem}
\begin{proof}
	Let $Y'\to X_{\overline{k}}$ be a finite \'etale morphism with $Y'$ a (normal) integral scheme over $\overline{k}$. Let $L/k$ be a finite field extension such that $Y'\to X_{\overline{k}}$ is defined over $L$, i.e., there is a finite \'etale morphism $Y\to X_L$ which is isomorphic to $Y'\to X_{\overline{k}}$ over $\overline{k}$. Now, by standard spreading out arguments
	\cite[17.7.8(ii)]{EGAIV4},
	there is a $\mathbb{Z}$-finitely generated regular subring $R\subseteq L$, a normal proper scheme $\mathcal{X}$ over $R$ with $\mathcal{X}_L\cong X$, a normal proper scheme $\mathcal{Y}$ over $R$ with $\mathcal{Y}_L\cong Y$ and a finite \'etale morphism $\mathcal{Y}\to \mathcal{X}$ extending $Y\to X_L$.

Since $X$ has the Hilbert property over $k$, we have that $X_L$ has the Hilbert property over $L$; see \cite[Proposition 3.2.1]{SerreTopicsGalois}. By the valuative criterion of properness, we have that $\mathcal{X}(R)^{(1)} = X_L(L)$. It follows that $\mathcal{X}$ has the Hilbert property over $R$ (by Definition \ref{DefinitionHPoverR}). Thus, by Theorem \ref{TheoremEtaleGroupTrivial}, the morphism $\mathcal{Y}\to \mathcal{X}$ is an isomorphism. This implies that $Y'\to X_{\overline{k}}$ is an isomorphism, as required.
\end{proof}

\section{The proof of Theorem \ref{TheoremFiniteFlatBaseChange}}

The persistence of the Hilbert property under finite flat extensions of finitely generated domains is a consequence of \cite[Proposition~3.2.1]{SerreTopicsGalois}.

\begin{proof}[Proof of Theorem \ref{TheoremFiniteFlatBaseChange}]
	Let $K$ and $L$ denote the fraction fields of $R$ and $S$, respectively.
	Define $A := X_S(S)^{(1)} \subseteq X_S(L)$, and assume that $A$ is thin in the geometrically irreducible variety $X_L$.
	By \cite[Proposition 3.2.1]{SerreTopicsGalois}, this implies that $A \cap X_K(K)$ is thin in $X_K$.
	Thus, to conclude the proof, it suffices to show that $A \cap X_K(K) \supseteq X(R)^{(1)}$. Given a near-$R$-integral point $x\in X(R)^{(1)}$,
	let $U\subseteq \Spec R$ be a dense open whose complement is of codimension at least two such that $x$ extends to a morphism $U\to X$.
	Let $V := \Spec S \times_{\Spec R} U$, and note that its complement is of codimension at least two since
	$\Spec S \to \Spec R$ is finite flat surjective. Since $V \to X$ is a near-$S$-integral point, this shows that $x\in X(S)^{(1)}$,
	and thus $x\in A\cap X_K(K)$,	as required.
\end{proof}

\bibliography{references}{}
\bibliographystyle{alpha}

\end{document}